\documentclass[10pt,a4paper]{scrartcl}
\usepackage[nochapters,eulermath]{classicthesis}
\usepackage[T1]{fontenc}
\usepackage{textcomp}
\usepackage{amsmath}
\usepackage{amssymb}
\usepackage{amsthm}
\usepackage{bm}

\usepackage[all,pdf]{xy}

\theoremstyle{plain}
\newtheorem{theorem}{Theorem}
\newtheorem{lemma}[theorem]{Lemma}
\newtheorem{proposition}[theorem]{Proposition}
\newtheorem{corollary}[theorem]{Corollary}

\newtheorem*{WISCinSet}{WISC (in $\set$)}

\theoremstyle{definition}
\newtheorem{definition}[theorem]{Definition}

\theoremstyle{remark}
\newtheorem{example}[theorem]{Example}
\newtheorem{remark}[theorem]{Remark}

\newcommand{\NN}{\mathbb{N}}
\newcommand{\NNplus}{\mathbb{N_+}}
\newcommand{\ZZ}{\mathbb{Z}}

\newcommand{\cC}{\mathcal{C}}
\newcommand{\cF}{\mathcal{F}}
\newcommand{\cG}{\mathcal{G}}

\newcommand{\into}{\hookrightarrow}    
\newcommand\onto{\twoheadrightarrow}    

\newcommand{\set}{\bm{\mathrm{set}}}
\newcommand{\Grp}{\bm{\mathrm{Grp}}}
\newcommand{\Top}{\bm{\mathrm{Top}}}

\newcommand{\id}{\mathrm{id}}
\newcommand{\pr}{\mathrm{pr}}

\newcommand{\ZON}{\mathcal{Z}}

\DeclareMathOperator{\im}{im}

\DeclareMathOperator{\Obj}{Obj}
\DeclareMathOperator{\Mor}{Mor}

\DeclareMathOperator{\lcm}{lcm}
\DeclareMathOperator{\dom}{dom}
\DeclareMathOperator{\Stab}{Stab}

\begin{document}

\title{\rmfamily\normalfont 
The weak choice principle WISC may fail in the category of sets}
\author{\spacedlowsmallcaps{David Michael Roberts}\footnote{Supported by the Australian Research Council 
(grant number DP120100106).
This paper will appear in the journal \emph{Studia Logica}.}
\\\small{\texttt{david.roberts@adelaide.edu.au}}}
\date{} 
\maketitle

\begin{abstract}
The set-theoretic axiom WISC states that for every set there is
a \emph{set} of surjections to it cofinal in \emph{all} such surjections. By
constructing an unbounded topos over the category of sets and using an
extension of the internal logic of a topos due to Shulman, we show that
WISC is independent of the rest of the axioms of the set
theory given by a well-pointed topos. This also gives an example of a topos
that is not a predicative topos as defined by van den Berg.
\end{abstract}

\section{Introduction}

Well-known from algebra is the concept of a \emph{projective object}: in a
finitely complete category this is an object $P$ such that any epimorphism
with codomain $P$ splits. The axiom of choice (AC) can be stated as saying
that every set is projective in the category of sets. Various constructive set
theories seek to weaken this, and in particular the axiom known as PAx
(Presentation Axiom) \cite{Aczel_78} or CoSHEP (Category of Sets Has Enough
Projectives) asks merely that every set $X$ has an epimorphism $P\onto X$
where $P$ is a projective set. Many results that seem to rely on the axiom of
choice, such as the existence of enough projectives in module categories, may
be proved instead with PAx. As a link with a more well-known axiom, PAx imples
the axiom of dependent choice.

There is, however, an even weaker option, here called WISC (to be explained momentarily). 
Consider the full subcategory $Surj/X \into \set/X$ of surjections with codomain $X$, in some category $\set$ of sets; clearly it is a large category. 
Then PAx implies the statement that $Surj/X$ has a \emph{weakly initial object}, namely an object with a map to any other object, not necessarily unique (the axiom of choice says $\id_X\colon X\to X$ is weakly initial in $Surj/X$). 
Another way to think of the presentation axiom is that for every set $X$ there is a `cover' $P \onto X$ such that any surjection $Y\onto P$ splits.

The axiom WISC (Weakly Initial Set of Covers), due to Toby Bartels and Mike Shulman, asks merely that the category $Surj/X$ has a weakly initial \emph{set}, for every $X$. 
This is a set $I_X$ of objects (that is, of surjections to $X$) such that for any other object (surjection), there is a map from \emph{some} object in $I_X$.
To continue the geometric analogy, this is like asking that there is a set of covers of any $X$ such that each surjection $Y\onto X$ splits locally over at least one cover in that set. 
An example implication of WISC is that the cohomology $H^1(X,G)$ defined by Blass in \cite{Blass_83} is indeed a set. 
The assertion that $H^1(X,G)$ is a proper class seems to be strictly weaker than $\neg$WISC, but to the author's knowledge no models have yet been produced
where this is the case.

The origin of the axiom WISC (see \cite{Roberts_12}) was somewhat geometric in flavour but the question naturally arises whether toposes, and in particular the category of sets, can fail to satisfy WISC. 
A priori, there is no particular reason why WISC should hold, so the burden is to supply an example where it fails. 
It goes without saying that neither AC nor PAx can hold in such an example.

The first result in this direction was from van den Berg (see \cite{vdBerg-Moerdijk_14}%
\footnote{In that paper, WISC is used in a guise of an
equivalent axiom called AMC, the Axiom of Multiple Choice. To avoid confusion
with other axioms with that name, this paper sticks with the term `'WISC'.})
who proved that WISC implies the existence of a proper class of regular cardinals, and so WISC must fail in Gitik's model of ZF \cite{Gitik_80}.
This model is constructed assuming the existence of a proper class of certain large cardinals, and it has no regular cardinals bigger than $\aleph_0$.
Working in parallel to the early development of the current paper, Karagila \cite{Karagila_14} gave a model of ZF in which there is a proper class of incomparable sets (sets with no injective resp.\ surjective functions between them) surjecting onto the ordinal $\omega$. 
This gave a large-cardinal-free proof that WISC was independent of the ZF axioms, answering a question raised by van den Berg.

The current paper started as an attempt to also give, via category-theoretic methods, a large-cardinals-free proof of the independence of WISC from ZF. 
Since the  release of \cite{Karagila_14}, this point is moot as far as independence from ZF goes. 
However, the proof in \cite{Karagila_14} relies on a symmetric submodel of a class-forcing model, which is rather heavy machinery. 
Thus this paper, while proving a slightly weaker result, does so with, in the opinion of the author, far less.

The approach we take is to consider the negation of WISC in the \emph{internal logic} of a (boolean) topos. 
This allows us to interpret the theory of a well-pointed topos together with $\neg$WISC. 
However, since this internal version of WISC holds in any Grothendieck topos (assuming for example AC in the base topos of sets) \cite{vdBerg-Moerdijk_14}, we necessarily consider a \emph{non-bounded} topos over the base topos of sets (recall that boundedness of a topos is equivalent to it being a Grothendieck topos). 
In fact the topos we consider is a variant on the `faux topos' mentioned in \cite[IV 2.8]{SGA4} (wherein `topos' meant what we now call a Grothendieck topos).

The reader familiar with such things may have already noticed that WISC or its negation is not the sort of sentence that can be written via the usual Kripke-Joyal semantics (see e.g.~\cite[\S VI.6]{MacLane-Moerdijk}) used for internal logic, as it contains unbounded quantifiers. 
As a result, we will be using an extension called the \emph{stack semantics},  given by Shulman \cite{Shulman_10}, that permits their use. 
The majority of the proof is independent  of the details of the stack semantics, which are only used to translate WISC from a statement in a well-pointed topos to a general topos (in fact a locally connected  topos, as this is the only case we will consider).

To summarise: starting from a well-pointed topos with natural number object we give a proper-class-sized group $\ZON$ equipped with a certain topology, and consider the topos $\ZON\set$ of sets with a continuous action of this group. 
Of course, the preceeding sentence needs to be formalised appropriately, and we do this in terms of a base well-pointed topos and a large diagram of groups therein.
We reduce the failure of WISC in the internal logic of $\ZON\set$ to simple group-theoretic statements. 
It should be pointed out that classical logic is used throughout, and all the toposes in this note are boolean.

Finally, the topos constructed as in the previous paragraph is not a \emph{predicative topos} as defined in \cite{vandenBerg_12}. 
These are analogues of toposes that should capture predicative mathematics, as toposes capture the notion of intuitionistic mathematics. 
This apparent failure is understood and carefully discussed in \emph{loc.~cit.}; the example given in this paper is hopefully of use as a foil in the development of predicative toposes.

The author's thanks go to Mike Shulman for helpful and patient discussions regarding the stack semantics.
Thanks are also due to an anonymous referee who found an earlier version of this paper contained some critical errors.

\section{WISC in the internal language}

We use the following formulation of WISC, equivalent to the usual statement in a well-pointed topos and due to Fran\c{c}ois Dorais \cite{Dorais_MO}.

\begin{WISCinSet}
  For every set $X$ there is a set $Y$ such that for every surjection  $q\colon Z\to X$ there is a map $s\colon Y \to Z$ such that $q\circ s \colon Y\to X$ is a surjection.
\end{WISCinSet}

The aim of this paper is to show that an internal version of $\neg$WISC is valid in the (non-well-pointed) topos constructed in section \ref{sec:construction} below. 
The internal logic of a topos, in the generality required here, is given by the \emph{stack semantics}.
We refer to \cite[section 7]{Shulman_10} for more details on the stack semantics, recalling purely what is necessary for the translation of WISC into the internal logic of a topos $S$ (Shulman takes weaker assumptions on $S$, but this extra generality is not needed here).

If $U$ is an object of $S$ we say that a formula of category theory $\phi$ with parameters in the category $S/U$ is a \emph{formula over $U$}. 
We have%
\footnote{Technically, this is only after choosing a splitting of the fibred category $S^\mathbf{2} \to S$, but in practice one only deals with a finite number of instances so this can be glossed over.}
the base change functor $p^*\colon S/U\to S/V$ for any map $p\colon V\to U$, and call the formula over $V$ given by replacing each parameter of $\phi$ by its image under $p^*$ the \emph{pullback} of $\phi$ (denoted $p^*\phi$). 
Note that the language of category theory is taken to be two-sorted, so there are quantifiers for both objects and arrows separately. Here and later $\onto$ denotes a map that is an epimorphism.

\begin{definition}{(Shulman \cite{Shulman_10})}\label{def:stack_semantics}
  Given the topos $S$, and a sentence $\phi$ over $U$, we define the relation $U\Vdash \phi$ recursively as follows
    \begin{itemize}
    \item
      $U\Vdash (f=g) \leftrightarrow f = g$
    \item
      $U\Vdash \top$ always
    \item
      $U\Vdash \bot \leftrightarrow U \simeq 0$
    \item
      $U\Vdash (\phi \wedge \psi) \leftrightarrow U\Vdash \phi$ and $U\Vdash \psi$
    \item
      $U\Vdash (\phi \vee \psi) \leftrightarrow U = V\cup W$, where 
      $i\colon V\into U$ and $j\colon W\into U$ are subobjects such that 
      $V\Vdash i^*\phi$ and $W\Vdash j^*\psi$
    \item
      $U\Vdash (\phi\Rightarrow \psi) \leftrightarrow$ for any $p\colon V\to U$
      such that $V\Vdash p^*\phi$, also $V\Vdash p^*\psi$
    \item
      $U\Vdash \neg \phi \leftrightarrow U\Vdash (\phi\Rightarrow \bot)$
    \item
      $U\Vdash (\exists X)\phi(X) \leftrightarrow 
      \exists p\colon V\twoheadrightarrow U$ and $A\in \Obj(S/V)$ 
      such that $V\Vdash p^*\phi(A)$
    \item
      $U\Vdash (\exists f\colon A\to B)\phi(f)\leftrightarrow 
      \exists p\colon V\twoheadrightarrow U$ and  
      $g\colon p^*A \to p^*B \in \Mor(S/V)$ such that $V\Vdash p^*\phi(g)$
    \item
      $U\Vdash (\forall X)\phi(X) \leftrightarrow$ for any $p\colon V\to U$ and 
      $A\in \Obj(S/V)$, $V\Vdash p^*\phi(A)$
    \item
      $U\Vdash (\forall f\colon A\to B)\phi(f)\leftrightarrow$ for any 
      $p\colon V\to U$ and $j\colon p^*A \to p^*B \in \Mor(S/V)$, 
      $V\Vdash p^*\phi(j)$
    \end{itemize}
    If $\phi$ is a formula over $1$ we say $\phi$ is \emph{valid} if $1\Vdash \phi$.
\end{definition}

Comparing with \cite[\S VI.6]{MacLane-Moerdijk} one can recognise the Kripke-Joyal semantics as a fragment of the above, where attention is restricted to monomorphisms rather than arbitrary objects in slice categories, and all quantifiers are bounded.

Since our intended model will be built using not just an arbitrary topos, but a locally connected and cocomplete one, the following lemma will simplify working in the internal logic. 
The proof follows that of lemma 7.3 in \cite{Shulman_10}. 
We recall that a locally connected topos $E$ is a topos over $\set$ with an additional left adjoint $\pi_0$ to the inverse image part of the global section functor, and an object $A$ is called \emph{connected} if $\pi_0(A)=1$.

\begin{lemma}\label{loc_conn_topos_forall}
  Let $E$ be a locally connected cocomplete topos. 
  Then then if for any \emph{connected} object $V$, arrow $p\colon V\to U$ and $A\in \Obj(S/V)$ we have $V\Vdash p^*\phi(A)$, then $U\Vdash (\forall X)\phi(X)$.
\end{lemma}

Here `locally connected cocomplete' is relative to a base topos $\set$ that is well-pointed (hence boolean) topos with natural number object (nno).
We will refer to the objects of $\set$ as `sets', but without an implication that these arise from a particular collection of axioms.
We will assume throughout that all toposes will come with an nno.

For a locally connected and cocomplete topos the statement of WISC translates, using definition \ref{def:stack_semantics} and applying lemma \ref{loc_conn_topos_forall}, into the stack semantics as follows:
\begin{equation}\label{int_WISC}
\begin{aligned}
  &\forall\ X\to U,\ U \text{ connected,} \\
  &\exists\ V\stackrel{p}{\onto} U,\ Y\to V,\\
  &\forall\ W\stackrel{q}{\to} V,\ W \text{ connected,}\ 
    Z\stackrel{g}{\onto}W\times_U X, \\
  &\exists\ T \stackrel{r}{\onto} W,\ T\times_V Y \xrightarrow{(\pr_1,l)} T\times_W Z, \\
  &\text{the map}\ T\times_V Y \xrightarrow{(\pr_1,l)} T\times_W Z \xrightarrow{r^*(g)} T\times_U X \text{ is an epi}.
\end{aligned}
\end{equation}

\noindent 
Note also that ``is an epi'' is a proposition whose statement in the stack semantics is equivalent to the external statement (see discussion around example 7.10 of \cite{Shulman_10}).
One does not need any knowledge of the stack semantics for the rest of this paper, and the uninitiated may choose to take (\ref{int_WISC}) as the \emph{definition} of WISC in the internal language of a locally connected cocomplete topos, and ignore the stack semantics entirely.

We will give a boolean $\set$-topos $E$ that is locally connected and cocomplete and in which the following statement, the negation of (\ref{int_WISC}), holds:
\begin{equation}\label{int_neg_WISC}
\begin{aligned}
  &\exists\ X\to U,\ U \text{ connected,} \\
  &\forall\ V\stackrel{p}{\onto} U,\ Y\to V,\\
  &\exists\ W\stackrel{q}{\to} V,\ W \text{ connected,}\ 
    Z\stackrel{g}{\onto}W\times_U X,\\
  &\forall\ T \stackrel{r}{\onto} W,\ T\times_V Y \xrightarrow{(\pr_1,l)} T\times_W Z, \\
  &\text{the map}\ T\times_V Y \xrightarrow{(\pr_1,l)} T\times_W Z \xrightarrow{r^*(g)} T\times_U X \text{ is not epi}.
\end{aligned}
\end{equation}

We denote the natural number object of $E$ by $\NN_d$, which is given by the image of the nno $\NN$ of $\set$ under the inverse image part of the geometric morphism $E\to \set$. 

\begin{proposition}\label{prop:implies_neg_WISC}
  In a connected, locally connected cocomplete topos $E$ such that $\pi_0$ reflects epimorphisms, the statement 
  \begin{equation}\label{neg_WISC_simple}
    \begin{aligned}
    &\forall\ Y\onto V,\ V \text{ connected,} \\
    &\exists\ \Omega \onto \NN_d \text{ with } \pi_0(\Omega) \simeq \pi_0(\NN_d),\\
    &\forall\ T \onto V,\ T \text{ connected, } T\times_V Y \xrightarrow{l} \Omega, \\
    & l \text{ is not epi}.
    \end{aligned}
  \end{equation}
  implies (\ref{int_neg_WISC}), the negation of WISC in the internal language of $E$.
\end{proposition}

\begin{proof}
  We give some facts about toposes that we will use in what follows.
  First, in a connected topos the terminal object is connected.
  Second, in a cocomplete topos one has infinitary extensivity, namely $A\times_B \coprod_{i\in I} C_i \simeq \coprod_{i\in I} A\times_B C_i$, and the initial object $0$ is \emph{strict}: any map to it is an isomorphism.
  Third, since $\pi_0$ is a left adjoint, it preserves epimorphisms.
  Combined with the hypothesis on $\pi_0$ this means a map $f$ in $E$ is an epimorphism if and only if $\pi_0(f)$ is an epimorphism. Similarly $\pi_0$ preserves initial objects and the hypotheses imply it also reflects initial objects.

  Now assume that (\ref{neg_WISC_simple}) holds in $E$. In (\ref{int_neg_WISC}) take $X\to U$ to be $\NN_d \to 1$ (using $1$ is connected). 
  Given an epimorphism $V\onto 1$, $V$ has a component as $\pi_0(V)\to 1$ is onto and $V = \coprod_{v\in\pi_0(V)} V_v$ (and $1$ is projective).
  Fix a component $V_0 \into V$.

  Given any $Y\to V$, take $Y_0 = V_0 \times_V Y$ to get $Y_0 \to V_0$. 
  If $Y_0$ is initial, then (\ref{int_neg_WISC}) can be seen to hold by taking $W = V_0$ and $g=\id$ since $T\times_V Y = T\times_{V_0} Y_0 = 0$ and as $r$ is an epi and $W$ is connected, $T\times \NN_d$ is not initial. 
  
  Hence we can assume $Y_0$ is not initial, and hence has at least one component and so $Y_0 \to V_0$ is an epi.
  Fix some $\Omega \onto \NN_d$ inducing an isomorphism $\pi_0(\Omega) \simeq \pi_0(\NN_d)$ such that the rest of (\ref{neg_WISC_simple}) holds.
  In (\ref{int_neg_WISC}) take $q$ to be the inclusion $V_0 \into V$ (hence $W=V_0$, which is connected), and $Z = V_0 \times \Omega$ with the epimorphism $g$ the product of $\id_{V_0}$ and $\Omega \onto \NN_d$.
  
  Now take any $T$ and pair of maps $T\onto V_0$ and $T\times_V Y = T\times_{V_0} Y_0 \xrightarrow{(\pr_1,l)} T\times_{V_0} Z = T \times \Omega$.
  We know that $T$ has a component by a similar argument to above, say $T_0 \into T$. 
  Then $T_0 \to V_0$ is epi so (\ref{neg_WISC_simple}) implies $T_0\times_{V_0} Y_0 = T_0 \times_V Y \to \Omega$ is not epi.
  This then implies $T_0 \times_V Y \to \Omega \to \NN_d$ is not epi, since if it were, $\pi_0(T_0\times_V Y) \to \pi_0(\Omega) \xrightarrow{\sim} \pi_0(\NN_d)$ would be epi, implying $\pi_0(T_0\times_V Y) \to \pi_0(\Omega)$ and hence $T_0 \times_V Y \to \Omega$ was epi. Thus there is some component of $\NN_d$ not in the image of this map, say indexed by $n\in \NN$.

 Then $T_0\times_V Y\to T_0\times \NN_d$ is not epi, as the component of $T_0\times \NN_d$ indexed by $n$ (isomorphic to $T_0$, which has $T_0 \to 1$ epi) is not in its image. It then follows that $T \times_V Y \to T\times \NN_d$ is not epi, and so (\ref{int_neg_WISC}) holds.
\end{proof}

\section{The construction}\label{sec:construction}

Given our base topos $\set$, we can consider the category of objects in $\set$
equipped with a linear order with no infinite descending chains, which we
shall call ordinals, in analogy with material set theory. 
The usual Burali-Forti argument---which requires no Choice---tells us there is a large category $O$ with objects ordinals and arrows the order-preserving injections onto initial segments. 
This large category is a linear preorder and has no infinite strictly descending chains. 
That there are multiple representatives for a particular order type, that is, non-identical isomorphic ordinals, does not cause any problems. 
We also note that $O$ has small joins (defined up to isomorphism in $O$).

Given a topological group $G$, the category of sets with a continuous $G$ action forms a cocomplete boolean topos $G\set$. In practice, one specifies a filter $\cF$ of subgroups of $G$ and then those $G$-sets all of whose stabiliser groups belong to $\cF$ are precisely those with a continuous action for the topology generated by $\cF$. 

For any group $G$, let $\cC$ be a collection of finite-index subgroups closed under finite intersections. 
Then there is a filter $\cF_\cC$ with elements those subgroups $H \leq G$ containing a subgroup appearing in $\cC$ (we say the filter is \emph{generated} by $\cC$). 
The category of continuous $G$-sets is then a full subcategory of the category of $G$-sets with finite orbits. 
The internal hom $Y^X$ is given by taking the set $\set(X,Y)$ then retaining only those functions whose stabiliser under the $G$-action $f\mapsto g\cdot \left(f(g^{-1}\cdot -)\right)$ belongs to $\cF_\cC$. 
The subobject classifier is the two-element set with trivial $G$-action.

\begin{remark}\label{rem:nice_cover}
  Notice that every transitive $G$-set $X$ that is continuous with respect to the topology given by $\cF_\cC$ (all $G$-sets will be assumed continuous from now on) has an epimorphism from some $G/L$ where $L\in \cC$. This is because any stabiliser $\Stab(x) \in \cF_\cC$, $x\in X$, is assumed to contain an element of $\cC$.
\end{remark}

\begin{example}\label{eg:bounded_depth_subgroups}
  For $\alpha$ an ordinal, let $\ZZ^\alpha$ be the set of functions $\alpha \to \ZZ$, considered as a group by pointwise addition. 
  Consider functions $d\colon \alpha \to \NNplus = \{1,2,3,\ldots\}$ such that $d(i) \not= 1$ for only finitely many $i\in \alpha$, which we shall call \emph{local depth functions}. 
  Such a function defines a subgroup $d\ZZ := \prod_{i\in\alpha}d(i)\ZZ \leq \ZZ^\alpha$ of finite index. 
  The intersection of two such subgroups, given by $d_1$ and $d_2$, is given by the function $i\mapsto \lcm\{d_1(i)d_2(i)\}$. 
  The subgroups belonging to the filter generated by this collection will be called \emph{bounded depth subgroups}. From now on $\ZZ^\alpha$ will be regarded as having the topology generated by this filter.
\end{example}

If we are given a split open surjection $p\colon H \to G$ (with $p$ and its splitting continuous) there is a geometric morphism $(p^*\dashv p_*)\colon H\set \to G\set$ with $p^*$ fully faithful and possessing a left adjoint $p_! \dashv p^*$. 
Here $p^*$ sends a $G$-set to the same set with the $H$-action via $p$ and $p_!(X) = X/\ker(p)$ with the obvious $G$-action.
The inverse image functor $p^*$ is in this case also a \emph{logical} functor, meaning that it preserves the subobject classifier and internal hom, as well as finite limits. 
In the case that $G$ is the trivial group: $p^*$ is denoted $(-)_d$ and sends a set to the same set with the trivial action; $p_!$ is denoted $\pi_0$ and $\pi_0(X)$ is the set of orbits of the $H$-action.

\begin{example}
  For $\alpha \into \beta$ ordinals, there is a split open surjection $\ZZ^\beta \to \ZZ^\alpha$, projection being given by restriction of the domain, and the splitting given by extending a function by $0$.
  Note that a local depth function on $\alpha$ gives a local depth function on $\beta$ by extending it by $1$.
\end{example}

Now consider a functor $\cG\colon O^{op}\to \Top\Grp_{sos}$, where $\Top\Grp_{sos}$ is the category of topological groups and split open surjections. 
Define the category $\cG\set$ with objects pairs $(\alpha,X)$ where $\alpha$ is an ordinal and $X$ is an object of $\cG(\alpha)\set$, and arrows $\cG\set((\alpha,X),(\beta,Y)) = \cG(\gamma)(X_\gamma,Y_\gamma)$ where $\gamma = \max\{\alpha,\beta\}$ and $X_\gamma,\ Y_\gamma$ are $X,Y$ considered as $\cG(\gamma)$-sets via the inverse image functors as above. 
The hom-sets are defined without making any choices since $O$ is a linear preorder, and so $\gamma$ is either $\alpha$ or $\beta$ (and we can take $\gamma = \alpha$ if $\alpha\simeq\beta$).
Composition is well defined due to the full faithfulness of the inverse image functors.
The objects of $\cG\set$ will be referred to as $\cG$-sets.
Informally, this category is the colimit of the large diagram of inverse image functors.

\begin{proposition}
  The category $\cG\set$ is a connected, locally connected, atomic and cocomplete boolean $\set$-topos. Moreover, $\pi_0$ reflects epimorphisms.
\end{proposition}

\begin{proof}
  Let us first show that we have a topos. 
  Finite limits exist because they can be calculated in any $\cG(\alpha)$ where $\alpha$ is greater than all ordinals appearing in the objects in the diagram, and when the universal property is checked in $\cG(\beta)$ for $\beta > \alpha$, the limit is preserved by the inverse image functor.
  Likewise the internal hom $(\alpha,X)^{(\beta,Y)}$ is defined as $X_\gamma^{Y_\gamma}$ in $\cG(\gamma)$ ($\gamma = \max\{\alpha,\beta\}$) and its universal property is satisfied due to inverse image functors preserving internal homs.
  The subobject classifier $\mathbf{2}$ in $\set$ is preserved by all inverse image functors $\set \to \cG(\alpha)\set$, so given any subobject in $\cG\set$ it has a classifying map to $\mathbf{2}$.
  Thus $\cG\set$ is a topos, and has a geometric morphism $((-)_d \dashv (-)^\cG)\colon \cG\set \to \set$ as it is locally small ($(-)^\cG := \cG\set(1,-)$ is the global points functor).
  It is easy to check there is a functor $\pi_0$ sending a $\cG(\alpha)$-set to its set of orbits and this is a left adjoint to $(-)_d$. 
  Thus $\cG\set$ is locally connected.
  Since $(-)_d$ is fully faithful and logical $\cG\set$ is also connected and atomic respectively.
  Small colimits can be calculated in $\cG(\alpha)$ where $\alpha$ is some small join of the ordinals appearing as the vertices of the diagram, and the universal property is verified since inverse image functors preserve all small colimits.
  Lastly, $\cG\set$ is boolean as $1 \to \mathbf{2} \leftarrow 1$ is a coproduct cocone, using the definition of colimits and the fact it is such in $\set$.

  To prove the last statement, suppose $X \to Y$ in $\cG\set$ (without loss of generality, take this in $\cG(\alpha)\set$ for some $\alpha$) is such that $\pi_0$ induces an epimorphism of connected components.
  Then for each orbit of $Y$ there is an orbit of $X$ mapping to it, and equivariant maps between orbits are onto, so $X\to Y$ is onto as a map of sets and hence an epi.
\end{proof}

The stack semantics in $\cG\set$ give a model of the structural set theory underlying $\set$, minus any Choice that may hold in $\set$ (see the discussion after lemma 7.13 in \cite{Shulman_10}). 
We will take a particular diagram of groups with the properties we need.

\begin{corollary}
  The diagram $\ZON\colon \alpha \mapsto \ZZ^\alpha$, where $\ZZ^\alpha$ is regarding as having the topology given by the filter of bounded depth subgroups, gives rise to a connected, locally connected boolean topos $\ZON\set$ such that $\pi_0$ reflects epimorphisms.
\end{corollary}

If one is working in a setting that permits such reasoning, the proper class-sized group to which the introduction alludes is the colimit over the inclusions $\ZON(\alpha) \into \ZON(\beta)$ given by the splittings, for $\alpha \into \beta$.
The rest of the paper will show that internal WISC fails in $\ZON\set$, and so WISC itself fails in the well-pointed topos given by the stack semantics of $\ZON\set$.

\section{The failure of WISC}\label{sec:failure_of_WISC}

We need some facts that hold in $\ZON\set$ regarding local depth functions. As a bit of notation, let us write $\ZON/d\ZZ$ for the transitive $\ZON$-set $\ZZ^\alpha/d\ZZ$ for $\alpha = \dom(d)$.

\begin{lemma}\label{local_depth_bounds}
  Let $\ZON/d_1\ZZ \to \ZON/d_2\ZZ$ be an equivariant map of $\ZON$-sets. 
  Then for every $i \in \alpha$ we have $d_2(i) \mid d_1(i)$.
\end{lemma}

\begin{proof}
  The existence of the map implies $d_1\ZZ$ is conjugate to a subgroup of $d_2\ZZ$, but all groups here are abelian so it \emph{is} a subgroup of $d_2\ZZ$.
  For the second statement, notice that the first statement implies $d_1(i)\ZZ \leq d_2(i)\ZZ \leq \ZZ$ for each $i\in \alpha$ and the result follows.
\end{proof}

We also need to consider what taking pullbacks looks like from the point of view of local depth functions.

\begin{lemma}\label{depth_of_fibred_product}
  Any orbit in
  \[
    \ZON/(d_1\ZZ \cap d_2\ZZ) \subset \ZON/d_1\ZZ \times_{\ZON/d_3\ZZ} \ZON/d_2\ZZ
  \]
  is isomorphic to a transitive $\ZON$-set with local depth function $d$ given by
  \[
    d(i) = \lcm\{d_1(i),d_2(i)\},\quad \forall i \in \alpha
  \]
  where $\alpha = \max\{\dom(d_1),\dom(d_2)\}$.
\end{lemma}

\begin{proof}
Notice that the fibred product as given is isomorphic to 
\[
  \prod_{i\in \alpha} \ZZ/d_1(i)\ZZ \times_{\ZZ/d_3(i)\ZZ} \ZZ/d_2(i)\ZZ
\]
where the $\ZZ^\alpha$ action is such that the $i^{th}$ coordinate---a copy of $\ZZ$---acts diagonally on the $i^{th}$ factor of the preceeding expression.
The stabiliser of any $(n_i,n'_i)_{i\in \alpha}$ is then the product of the stabilisers of the $\ZZ$-action of the various $\ZZ/d_1(i)\ZZ \times_{\ZZ/d_3(i)\ZZ} \ZZ/d_2(i)\ZZ$.
We thus only need to consider the simpler problem of determining the stabilisers for a $\ZZ$-set $\ZZ/k\ZZ \times_{\ZZ/m\ZZ} \ZZ/l\ZZ$.

The stabiliser of $(0,0)$ is $\ZZ/(k\ZZ\cap l\ZZ)$, from which the result follows by the description in example \ref{eg:bounded_depth_subgroups} of the intersection of subgroups given by local depth functions.
We only then need to consider the stabilisers of $(0,n)$ for $n\in \ZZ/l\ZZ$ as all others are equal to one of these by abelianness -- but $\Stab(0,n)$ is again $\ZZ/(k\ZZ\cap l\ZZ)$ using abelianness.
The statement regarding local depth functions then follows.
\end{proof}

We need a special collection of subgroups of $\ZZ^\alpha$ in the proof of theorem \ref{WISC_fails_in_Zset} below, namely those given by local depth functions $\delta[\alpha,n,i]\colon \alpha\to \NNplus$ defined as
\[
  \delta[\alpha,n,i](k) = 
    \begin{cases}
      n & \text{if $k = i$;}\\
      1 & \text{if $k \not= i$.}
    \end{cases}
\]
Note that the transitive $\ZON$-set $\ZON/\delta[\alpha,n,i]\ZZ$ has underlying set $\ZZ/n\ZZ$, and that $\Omega[\alpha,i] := \coprod_{n\in \NNplus} \ZON/\delta[\alpha,n,i]\ZZ$ is an object of $\ZON\set$ for any $\alpha\in O$ and $i\in\alpha$.

\begin{theorem}\label{WISC_fails_in_Zset}
  The statement of WISC in the stack semantics in $\ZON\set$ fails.
\end{theorem}

\begin{proof}
In the notation of proposition \ref{prop:implies_neg_WISC}, taking transitive $\ZON$-sets for connected objects, we need to show that for any $Y\onto \ZON/H$, there is an $\Omega$ such that for any $r\colon \ZON/K \to \ZON/H$, any $l\colon \ZON/K\times_{\ZON/H} Y \to \Omega$ is not an epimorphism. 

Let us write $Y = \coprod_{y\in\pi_0(Y)} Y_y$, and note that this coproduct, like all colimits in $\ZON\set$ takes place in some $\ZZ^\alpha\set$.
In particular, by remark \ref{rem:nice_cover} each $Y_y$ has an epimorphism 
from some $\ZON/d_y\ZZ$ for a local depth function $d_y \colon \alpha \to \NNplus$. 
As a result $H \leq \ZZ^\alpha$, so fix some $d_H\colon \alpha \to \NNplus$ to get an epimorphism $\ZON/d_H\ZZ \to \ZON/H$.
Define $\Omega = \Omega[\alpha+1,\top_{\alpha+1}]$, where $\top_{\alpha+1}$ is the top element of the ordinal $\alpha+1$.
Given $\ZON/K \to \ZON/H$, fix a local depth function $d_K\colon \beta \to \NNplus$ such that $d_K\ZZ \leq K$ (without loss of generality, we can assume $\alpha \leq \beta$).

Since $\ZON\set$ is infinitary extensive, we have 
\[
  \ZON/K\times_{\ZON/H} Y \simeq \coprod_{y\in\pi_0(Y)} \ZON/K\times_{\ZON/H} Y_y.
\] 
Any map $l\colon \ZON/K\times_{\ZON/H} Y \to \Omega$ is then given by a collection of maps $l_y \colon \ZON/K\times_{\ZON/H} Y_y \to \Omega$.
We need to show that this collection of maps is not jointly surjective, and will do this by showing the image of $l_y$, for arbitrary $y$, must be contained in a strict subobject of $\Omega$ that is independent of $y$.

Given an epimorphism $\ZON/d_y\ZZ \to Y_y$, consider, in $\ZON/d_K\ZZ \times_{\ZON/d_H\ZZ} \ZON/d_y\ZZ$, an orbit $\ZON/\delta_y\ZZ$ where $\delta_y(i) = \lcm\{d_K(i),d_y(i)\}$ for each $i\in \beta$, by lemma \ref{depth_of_fibred_product}.
In particular, we have that $\delta_y(\top_{\alpha+1}) = d_K(\top_{\alpha+1}) =: N_0$ is independent of $y$.

Compose the inclusion $\ZON/\delta_y\ZZ \into \ZON/K \times_{\ZON/H} Y_y$with $l_y$ to get a map 
\[
  l'_y\colon \ZON/\delta_y\ZZ \to \Omega = \coprod_{n\in \NNplus} \ZON/\delta[\alpha,n,i]\ZZ.
\]
Applying lemma \ref{local_depth_bounds} to this map with $i = \top_{\alpha+1}$ we find that $n\mid N_0$ for any $n$ such that $\ZON/\delta[\alpha,n,i]\ZZ \subset \im l'_y$. Thus the image of any $l_y$ and hence of $l$ is contained in 
\[
  \coprod_{n\mid N_0} \ZON/\delta[\alpha,n,i]\ZZ \subsetneqq \Omega,
\]
hence $l$ is not an epimorphism.
\end{proof}

Recall that ETCS is a set theory defined by specifying the properties of the
category of sets \cite{Lawvere_64}, namely that it is a well-pointed topos
(with nno) satisfying the axiom of choice. We can likewise specify a
choiceless version, which is the theory of a well-pointed topos (with nno).
Given a model $\set$ of ETCS, we have constructed a well-pointed topos in
which WISC is false. Thus we have our main result.

\begin{corollary}
  Assuming ETCS is consistent, so is the theory of a well-pointed topos with nno
  plus the negation of WISC.
\end{corollary}

Finally, we recall the definition from \cite{vandenBerg_12} of a predicative
topos: this is a \mbox{$\Pi W$-pretopos} satisfying WISC (or, as called there, AMC).

\begin{corollary}
  The topos $\ZON\set$ is not a predicative topos.
\end{corollary}


\end{document}